\documentclass[a4paper,12pt]{extarticle}
\usepackage{latexsym}
\usepackage{amscd}
\usepackage{graphics}
\usepackage{amsmath}
\usepackage{amssymb}
\usepackage{mathrsfs}
\usepackage{amsthm}
\usepackage{xcolor}
\definecolor{will}{RGB}{2,167,179}
\usepackage{accents}
\usepackage{enumerate}
\usepackage{url}
\usepackage{tikz-cd}
\usepackage{marginnote}
\usepackage{hyperref}
\hypersetup{
    colorlinks=true,
    linkcolor=will,      
    urlcolor=will,
    citecolor=will,
  }

\theoremstyle{plain}
\newtheorem{theorem}{\sc Theorem}[section]
\newtheorem{prop}[theorem]{\sc Proposition}
\newtheorem{lem}[theorem]{\sc Lemma}
\newtheorem{cor}[theorem]{\sc Corollary}
\theoremstyle{definition}
\newtheorem{defn}[theorem]{\sc Definition}
\newtheorem{rem}[theorem]{\sc Remark}

\makeatletter
\DeclareFontFamily{OMX}{MnSymbolE}{}
\DeclareSymbolFont{MnLargeSymbols}{OMX}{MnSymbolE}{m}{n}
\SetSymbolFont{MnLargeSymbols}{bold}{OMX}{MnSymbolE}{b}{n}
\DeclareFontShape{OMX}{MnSymbolE}{m}{n}{
    <-6>  MnSymbolE5
   <6-7>  MnSymbolE6
   <7-8>  MnSymbolE7
   <8-9>  MnSymbolE8
   <9-10> MnSymbolE9
  <10-12> MnSymbolE10
  <12->   MnSymbolE12
}{}
\DeclareFontShape{OMX}{MnSymbolE}{b}{n}{
    <-6>  MnSymbolE-Bold5
   <6-7>  MnSymbolE-Bold6
   <7-8>  MnSymbolE-Bold7
   <8-9>  MnSymbolE-Bold8
   <9-10> MnSymbolE-Bold9
  <10-12> MnSymbolE-Bold10
  <12->   MnSymbolE-Bold12
}{}
\let\llangle\@undefined
\let\rrangle\@undefined
\DeclareMathDelimiter{\llangle}{\mathopen}%
                     {MnLargeSymbols}{'164}{MnLargeSymbols}{'164}
\DeclareMathDelimiter{\rrangle}{\mathclose}%
                     {MnLargeSymbols}{'171}{MnLargeSymbols}{'171}
\title{Maximum principles in symplectic homology}
\date{\today}
\author{Will J. Merry and Igor Uljarevic}
\begin{document}
\maketitle{}

\begin{abstract}
In the setting of symplectic manifolds which are convex at infinity, we use a version of the Aleksandrov maximum principle to derive uniform estimates for Floer solutions that are valid for a wider class of Hamiltonians and almost complex structures than is usually considered. This allows us to extend the class of Hamiltonians which one can use in the direct limit when constructing symplectic homology. As an application, we detect elements of infinite order in the symplectic mapping class group of a Liouville domain and prove existence results for translated points.
\end{abstract}

\section{Introduction}
\label{sec:intro}

Symplectic homology is a natural generalisation of Floer homology to open symplectic manifolds which are  convex at infinity. It is constructed via a direct limit of Floer homology groups for Hamiltonians which get steeper and steeper at infinity. This construction -- in the case where the symplectic manifold is exact -- is due to Viterbo \cite{Viterbo1999} -- although there were various previous flavours and incarnations. It has since been generalised to non-exact symplectic manifolds by Ritter \cite{Ritter2010,Ritter2013,Ritter2014}. \\

Since the underlying symplectic manifold is non-compact, there are additional technical difficulties in the construction of symplectic homology which are not present in the closed setting. Chief amongst these is the necessity  of a \textbf{maximum principle} to prevent Floer solutions from escaping to infinity. The need for such a maximum prinicple to hold severely limits the type of Hamiltonians which one can take in the aforementioned direct limit. The standard technique is to use Hamiltonians that are radial and linear on the convex end. More precisely, the assumption that our underlying symplectic manifold $(M,\omega)$ is convex at infinity means that there is a compact domain $M_1 \subset M$ such that the complement of $M_1$ is symplectomorphic to the positive part of the symplectisation of a closed contact manifold:
\begin{equation}
\label{eq:convex_at_infinity}
(  M \setminus M_1 , \omega ) \cong (\Sigma \times (1, \infty), d(r \alpha)),
\end{equation}
where $(\Sigma,\alpha)$ is a closed contact manifold and $r$ denotes the coordinate on $[1, \infty)$. Then one works with Hamiltonians $H_t$ on $M$ that are of the form 
\begin{equation}
\label{eq:linearham}
  H_t(x,r) = ar +b, 
\end{equation}
for large $r$. For such Hamiltonians the $r$-component of a solution of the Floer equation is necessarily subharmonic, and this prevents such solutions from escaping to infinity. \\

In \cite{Ritter2016},  Ritter extended the class of Hamiltonians to those which are of the form 
\begin{equation}
\label{eq:1homo_ham}
  H_t(x,r) = r h_t (x)  + b,
\end{equation}
for large $r$, where $h_t : \Sigma \to \mathbb{R}$ is a  function which is \textbf{invariant under the Reeb flow} of $\alpha$.  A similar construction was also given independently by the second author in \cite{Uljarevic2015}, and for time-independent $h$ a different proof was given by Fauck in his thesis \cite{Fauck2016}.

In this paper we further generalise this and remove the hypothesis that $h_t$ must be invariant under the Reeb flow. Namely, we prove:

\begin{theorem}[Extended maximum principle]
\label{prop:emp}
The maximum prinicple applies to \emph{\textbf{all}} Hamiltonians of the form \eqref{eq:1homo_ham}.
\end{theorem}

\begin{rem}
Theorem \ref{prop:emp} is stated more precisely in Theorem \ref{thm:maxprin} below. Actually, the existence of a maximum principle depends not only the choice of Hamiltonian but also on the choice of almost complex structure (for simplicity we suppress this fact during the Introduction). The standard maximum principle requires one to work with an almost complex structure which is of SFT-type at infinity. In addition to extending the class of Hamitonians for which a maximum principle is applicable, we also extend the class of complex structures to those of \textbf{twisted SFT-type} at infinity (see Definition \ref{defn:tsft} below). As far as applications are concerned, extending the class of almost complex structures is just as important as extending the class of Hamiltonians -- see for instance Proposition \ref{prop:naturality} below).
\end{rem} 

We believe that Theorem \ref{prop:emp} will be useful in several situations. In the present paper we discuss two of them -- in the \hyperlink{out}{Outlook} at the end of this section we mention further applications that will appear in the sequel to this paper. \newline

Assume now that $(M, \omega)$ is a symplectic manifold which is convex at infinity as above, and for which the symplectic homology $\mathrm{SH}_*(M)$ is well defined. This really is an extra assumption: the minimum one needs to assume that $(M, \omega)$ is \textbf{$\text{weakly}^{+}$ monotone}, a condition introduced by Hofer and Salamon in \cite{HoferSalamon1995}, which requires that \textbf{at least one} of the following three conditions is fulfilled:
\begin{itemize}
  \item $\omega|_{\pi_2(M)} = 0$ or $c_1|_{\pi_2(M)} = 0$,
  \item $\omega$ is positively monotone: there exists $ \beta >0$ such that $(\omega - \beta c_1) |_{\pi_2(M)} =0$,
  \item the minimal Chern number $N$ of $(M^{2n},\omega)$ is at least $n-1$.
\end{itemize}
See Ritter \cite{Ritter2014} for details of the construction in this setting. At the other end of the spectrum, the ``simplest'' type of symplectic manifolds for which symplectic homology can be defined are \textbf{Liouville domains}. Here one starts with a compact exact symplectic manifold $(M_1, \omega_1 = d \lambda_1)$ with boundary $\Sigma = \partial M_1$, which in addition has the property that the vector field $Y_1$ defined implicitly by requiring that $\lambda_1 = \omega(Y_1, \cdot)$ points strictly outwards along $\Sigma$. Consequently there is a smooth function $r$ defined on a neighbourhood of $\Sigma$ in $M_1$ with values in $(0,1]$ such that $\Sigma = r^{-1}(1)$ and such that $d r (Y_1) = r$. This allows us to identify a neighbourhood of $\Sigma$ in $M_1$ with $ \Sigma \times (0,1]$. One then ``completes'' $(M_1, \omega_1)$ into a non-compact symplectic manifold $(M,\omega)$ by setting 
\[
M := M_1 \cup_{\Sigma}(  \Sigma \times \mathbb{R}^+ ),
\] 
We then extend $\lambda_1$ to a one-form $\lambda$ on $M$ by setting $\lambda = r \alpha$ on  $  \Sigma \times \mathbb{R}^+$. Setting $\omega = d \lambda$, the manifold $(M,\omega)$ is then obviously convex at infinity. This is the class of symplectic manifolds Viterbo \cite{Viterbo1999} originally worked with. \newline

Symplectic homology is oftentimes an infinite dimensional theory. The most famous early result, originally proved by Viterbo \cite{Viterbo1996} (and independently by Salamon-Weber \cite{SalamonWeber2006}, Abbondandolo-Schwarz \cite{AbbondandoloSchwarz2006}, and Abouzaid \cite{Abouzaid2015}) is that the symplectic homology of a cotangent bundle  $(T^*N, d \lambda_{\mathrm{can}})$ equipped with its standard symplectic form is equal to the singular homology of the free loop space:
\[
  \mathrm{SH}_*(T^*N) \cong \mathrm{H}_*(\Lambda(N); \mathbb{Z}_2).
\]
Nevertheless, motivated by the Seidel representation \cite{Seidel1997}, in \cite{Ritter2014} Ritter discovered that the existence of a \textbf{loop} of Hamiltonian diffeomorphisms with \textbf{positive} slope at infinity forces symplectic homology to be finite dimensional. Here is the precise statement.

\begin{theorem}[Ritter \cite{Ritter2014}]
\label{thm:sh_is_small}
Assume $(M,\omega)$ is a $\text{weakly}^+$ monotone symplectic manifold which is convex at infinity. If there exists a loop of Hamiltonian diffeomorphisms generated by a Hamiltonian $H_t : M \to \mathbb{R}$ which for $r$ large is of the form 
\[
  H_t(x,r) =r h_t(x) + b, \qquad h_t >0,
\]
such that
\begin{equation}
\label{inv}
h_t \text{ is invariant under the Reeb flow,}
\end{equation}
then the summand of symplectic homology coming from contractible orbits can be seen as a localisation (i.e quotient) of the quantum homology of $M$. 
\end{theorem}

The only reason the $h_t$ was required to be invariant under the Reeb flow was that this was the only class for which the maximum principle was known to hold. Our Theorem \ref{prop:emp} removes this restriction.

\begin{cor}
The assumption \eqref{inv} in Theorem \ref{thm:sh_is_small} is not necessary.
\end{cor}

We conclude this Introduction by discussing applications of our maximum principle. The first extends previous work of the second author \cite{Uljarevic2015} and resolves a  conjecture a Biran-Giroux \cite{BiranGiroux2005}. Assume now that $(M,\omega = d \lambda)$ is the completion of a Liouville domain $(M_1, \omega_1 = d \lambda_1)$ with boundary $\Sigma = \partial M_1$. Consider the  subgroup 
\[
  \mathcal{G} := \left\{ \phi \in \mathrm{Symp}(M, \omega) \mid \phi^* \lambda - \lambda  = d f, \ \text{supp}(f) \subset \mathrm{int}(M_1)) \right\},
\]
and the subset
\[
  \mathcal{G}_c : = \left\{ \phi \in \mathcal{G} \mid \mathrm{supp}(\phi) \subset \mathrm{int}(M_1) \right\}.
\]
Clearly $\mathcal{G}_c$ is a subgroup of the group $\mathrm{Symp}_c(M_1, \omega_1)$ of compactly supported symplectomorphisms of $(M_1, \omega_1)$. In fact, the inclusion $\mathcal{G}_c \hookrightarrow \mathrm{Symp}_c(M ,\omega)$ is a homotopy equivalence. This is due to Biran-Giroux \cite{BiranGiroux2005}; a  proof can be found in \cite[Lemma 3.3]{Uljarevic2015}. \\

Next, if we denote by $\alpha$ the restriction of the one-form $\lambda_1$ on $M_1$ to the boundary $\Sigma$, then $\alpha$ is a contact form on $\Sigma$. Let us denote the associated contact distribution by $\xi$, and the group of contactomorphisms by $\mathrm{Cont}(\Sigma,\xi)$. An element $\phi \in \mathcal{G}$ naturally defines an element of $\mathrm{Cont}(\Sigma, \xi)$. Indeed, for such a $\phi$ there exists $T > 0$ such that on $\Sigma \times (T, \infty)$, one has
\[
  \phi(x,r) = \left(\varphi(x), \frac{r}{\kappa(x)} \right),
\]
where $\varphi \in \mathrm{Cont}(\Sigma, \xi)$ and $\kappa : \Sigma \to (0, \infty)$ satisfies $ \varphi^*\alpha = \kappa \alpha$. Thus there is a well defined map, called the \textbf{ideal restriction map}, given by
\[
  \Theta : \mathcal{G} \to \mathrm{Cont}(\Sigma, \xi), \qquad \Theta(\phi) := \varphi.
\]
Moreover, Biran and Giroux discovered \cite{BiranGiroux2005} that the map $\Theta$ is a Serre fibration over $\mathrm{Cont}(\Sigma, \xi)$ with fibre $\mathcal{G}_c$. This implies there is a long exact sequence 
\[
  \begin{tikzcd}  
 \pi_k( \mathrm{Symp}_c(M_1 ,\omega_1)) \ar[r] & \pi_k(\mathcal{G}) \ar[r] & \pi_k(\mathrm{Cont}(\Sigma, \xi)) \ar[r, "\Delta"] &\pi_{k-1}(\mathrm{Symp}_c(M_1,\omega_1))
   \end{tikzcd}
\]
Denoting by 
\[
  \mathrm{SMCG}(M_1 ,\omega_1) := \pi_0 (\mathrm{Symp}_c(M_1,\omega_1)), 
\]
the \textbf{symplectic mapping class group} of $(M_1, \omega_1)$, the connecting homomorphism 
\[
  \Delta : \pi_1( \mathrm{Cont}(\Sigma, \xi)) \to \mathrm{SMCG}(M_1,\omega_1) 
\]
is given as follows. Given a \textbf{loop} $  \varphi^t : \Sigma \to \Sigma $, $t \in S^1$,  of contactomorphisms based at the identity, choose a lift $[0,1] \ni t \mapsto \phi^t \in \mathcal{G}$ such that $\Theta(\phi^t) = \varphi^t$. Then set 
\begin{equation}
\label{eq:Delta}
  \Delta( [ \varphi^t]) := [ \phi^1],
\end{equation}
where in both cases $[ \cdot]$ denotes the appropriate equivalence relation. Explicitly, if $\varphi^t$ has \textbf{contact Hamiltonian} $h_t : \Sigma \to \mathbb{R}$ (see \eqref{eq-Xh} below for the definition of $h_t$), define $H_t :M \to \mathbb{R}$ by first setting
\[
   H_t(x,r) = \beta(r) r h_t(x),\qquad \forall \, (x,r) \in \Sigma \times \mathbb{R}^+, 
\]
where 
\begin{equation}
\label{eq:beta}
  \beta : (0, \infty) \to [0, 1], \qquad \beta(r) = \begin{cases}
  0, & r \in (0,1/4), \\
  1, & r \in (3/4, \infty),
  \end{cases}
   \qquad \beta'(r) \ge 0,
\end{equation}
 and then extending $H_t$ to all of $M$ by setting $H_t = 0$ on the rest of $M$. Then if $\phi_H^t$ denotes the Hamiltonian flow of $H$, one has $\Theta(\phi_H^t) = \varphi^t$, and hence $ \Delta([\varphi^t]) = [ \phi^t_H]$.

Biran and Giroux conjectured that if the symplectic homology of $M$ is sufficiently rich and $\varphi^t $ is a \textbf{positive} loop (this means that the contact Hamiltonian $h_t$ is everywhere positive) then the element $\Delta( [ \varphi^t])$ should be of \textbf{infinite order} inside $\mathrm{SMCG}(M_1,\omega_1)$. In the following theorem, the total dimension of the symplectic homology of $M$ is denoted by $\dim \mathrm{SH}_*(M)$.

 \begin{theorem}
  \label{thm:infinite_order}
  Let $(M_1^{2n}, \omega_1)$ be a Liouville domain with completion $(M,\omega)$. Let $\varphi_t:\Sigma\to\Sigma$ be a positive loop of contactomorphisms on the boundary $\Sigma:=\partial M.$ If
\[ 
\sum_{k = 1}^{2n} \dim \mathrm{H}_k(M_1; \mathbb{Z}_2) <  \dim \mathrm{SH}_*(M)
 \]
 then the element $\Delta([\varphi^t])$ from \eqref{eq:Delta} is of infinite order in the symplectic mapping class group of $(M_1, \omega_1)$.
  \end{theorem} 
As remarked above, a special case of the theorem, in which the loop $\varphi^t$ is strict, was proved in \cite{Uljarevic2015}. In addition, recently progress has been made in several related directions by other authors: Seidel \cite{Seidel2014}, Chiang, Ding and van Koert  \cite{ChiangDingvanKoert2014,ChiangDingvanKoert2016}, and Barth, Geiges and Zehmisch \cite{BarthGeigesZehmisch2016}. See Remarks 4.1 and 4.6 in \cite{BarthGeigesZehmisch2016} for a more detailed historical overview.

% \begin{rem}
% Suppose the Reeb flow of $( \Sigma, \alpha)$ is periodic. Then a particular choice of loop $ \varphi^t$ of strict contactomorphisms is the Reeb flow itself. The corresponding element one obtains in the symplectic mapping class group is known as a \textbf{fibred Dehn twist}. In \cite{ChiangDingvanKoert2014,ChiangDingvanKoert2016}, Chiang, Ding and van Koert proved that (under certain technical assumptions), fibred Dehn twists are always of infinite order in the symplectic mapping class group. This was then generalised by Barth, Geiges and Zehmisch \cite{BarthGeigesZehmisch2016} to show that any non-trivial composition of fibred Dehn twists is of infinite order in the symplectic mapping class group.
% \end{rem}

Another application of the maximum principle pertains to \textbf{translated points}. Given a contactomorphism $ \varphi : \Sigma \to \Sigma$ of a contact manifold $ ( \Sigma, \alpha)$, a translated point of $ \varphi$ is a point $x$ such that $ \varphi(x)$ and $x$ both lie on the same Reeb orbit of $ \alpha$, and such that $ \varphi^* \alpha|_{x} = \alpha_x$. This notion was introduced by Sandon in \cite{Sandon2012}, although it can be seen as a special case of a leaf-wise intersection point as introduced by Moser in \cite{Moser1978}. Using Rabinowitz Floer homology, the first author and Albers proved:
\begin{prop}[ \cite{AlbersMerry2013a}]
Let $(M_1^{2n}, \omega_1)$ be a Liouville domain with completion $(M,\omega)$. Suppose $ \mathrm{SH}_*(M)$ is infinite dimensional. Then any contactomorphism of the boundary isotopic to the identity through contactomorphisms either has infinitely many translated points or a translated point on a closed Reeb orbit.
\end{prop}
In Section \ref{sec:tp} we give a new proof of this result using the new maximum principle.\newline

\hypertarget{out}{\textit{Outlook:}} The present paper concludes by briefly explaining how to define the analogue of \textbf{positive symplectic homology} in this setting. This gives rise to an invariant $ \mathrm{SH}^+( \varphi^t, M)$ associated to a path of contactomorphisms on the boundary of a Liouville domain, which we call the \textbf{postive symplectic homology for the contact isotopy $ \varphi^t$}. It is defined whenever $ \varphi^1 \ne \mathrm{Id}$, and, up to a grading shift, depends only on the time-1 map $ \varphi^1$. For $ \varphi^t$ equal to small piece of the Reeb flow this group agrees with the usual positive symplectic homology  $ \mathrm{SH}^+(M)$. In general though this is not the case, and the groups $ \mathrm{SH}^+( \varphi^t, M)$ can be used to obtain more refined existence results on translated points. We will study these groups thoroughly in a sequel to the present paper.

A further variant on this idea is to construct groups $ \mathrm{SH}(U,M)$ associated to an open set $ U $ of the boundary $ \Sigma$, by using contact Hamiltonians associated to contact isotopies that are supported inside $U$. This gives applications to contact embedding and squeezing problems, and will again be discussed in this paper's sequel. \newline

Finally, we remark that Groman \cite{Groman2015} has recently developed a novel approach to obtaining $L^{\infty}$-bounds for Floer solutions on non-compact manifolds. It would be interesting to see whether his methods can recover the results proved here.  \newline

\emph{Acknowledgement.} We thank Peter Albers, Paul Biran, Leonid Polterovich and Alex Ritter for several illuminating discussions during the preparation of this article. The second author was partially supported by the European Research Council (ERC) under the European Union's Horizon 2020 research and innovation programme, starting grant No.~637386.

\section{Preliminaries on symplectisations}
\label{sec-preliminaries}
Although we are primarily interested in non-compact manifolds that are convex at infinity, almost all of our arguments will take place in the symplectisation of a closed contact manifold. Thus we begin by fixing some notation in this setting.

Suppose $ ( \Sigma, \xi)$ is a closed coorientable contact manifold of dimension $2n-1$. Let $S \Sigma := \Sigma \times \mathbb{R}^+$ denote the \textbf{symplectisation} of $ \Sigma$. A choice of contant form $ \alpha$ on $ \Sigma$ supporting $ \xi$ gives rise to a symplectic form $ d \lambda$ on $ S \Sigma$. The one-form $ \lambda$ is defined by $ \lambda = r \alpha$, where $ r $ is the coordinate on $ \mathbb{R}^+$. Given $ r  >0 $, we abbreviate by $S\Sigma(r)$ the non-compact open manifold $ \Sigma \times ( r, \infty)$. 

We denote by $Y = r \partial_r$ the \textbf{Liouville vector field} on $  S \Sigma$, and we denote by $R$ the \textbf{Reeb vector field} on $ \Sigma$, which we also think of as a vector field on $ S \Sigma$. Note that the symplectic complement of the 2-plane field spanned by $R$ and $Y$ is exactly the contact distribution $ \xi$.

Suppose $j$ is an almost complex structure on the symplectic vector bundle $( \xi ,d \alpha) \to \Sigma$ which is compatible\footnote{Throughout we use the sign convention that an almost complex structure $I$ is compatible with a symplectic form $ \Omega$ if $ \Omega(I \cdot, \cdot)$ is a positive definite symmetric form.} with $ d \alpha$. The almost complex structure $j$ uniquely determines an almost complex structure $J$ on $S \Sigma$ itself, which is defined by 
\begin{equation}
\label{eq-SFT}
J ( aR + bY + v) = aY - b R + jv, \qquad a, b \in \mathbb{R}, \qquad v \in \xi.
\end{equation}
We say that $J$ is of \textbf{SFT-type}.\\ 

We now generalise this notion. Suppose we are given a smooth positive functions $q : \Sigma \to \mathbb{R}^+$. Let $ Q : S \Sigma \to \mathbb{R}^+$ be defined by 
$$Q(x,r) = rq(x).$$
Observe that the Hamiltonian vector field\footnote{Our sign convention is that $ d \lambda(X_Q, \cdot) = -dQ$.} $X_Q$ of $Q$ takes the form
\begin{equation}
\label{eq-XbQ}
X_Q(x,r) = X_q(x) - dq(x)R(x)Y(r),
\end{equation}
where $X_q$ is the vector field on $ \Sigma$ defined by 
$$ \alpha (X_q) = q, \qquad d \alpha(X_q, \cdot) = dq(R) \alpha - d q.$$
Let $N_Q$ denote the symplectic complement of the 2-plane field spanned by $X_Q$ and $Y$ in $T (S \Sigma)$. Explicitly, 
$$ N_Q(x,r) = \left\{ \left( v, -\frac{rdq(x)v}{q(x)} \right)  \mid v \in \xi_x\right\}.$$
Thus given $ v \in \xi_x$ there is a unique vector $ \zeta_Q(v) \in N_Q(x,r)$ such that the first component of $ \zeta_Q(v)$ is $v$.
\begin{defn}
\label{defn:tsft}
Let $q \in C^{ \infty}( \Sigma, \mathbb{R}^+)$ denote a smooth positive function, and set $Q = rq$ as above. Let $j$ denote a compatible almost complex structure on $ \xi$. The pair $(q,j)$ determines an almost complex structures\footnote{This notation is slightly inconsistent, since $J_Q$ depends on both $Q$ and $j$. In general we will suppress the dependence of $j$ in all our notation.} $J_Q$ on $ S \Sigma$ which, following Albers-Frauenfelder \cite{AlbersFrauenfelder2010} we will call of \textbf{twisted SFT-type}, via the formula:
\begin{equation}
\label{eq-twisted-SFT}
J_Q( a X_Q + b Y + \zeta_Q(v)) = bX_Q - a Y + \zeta_Q(jv), \qquad a,b \in \mathbb{R}, \qquad v \in \xi.
\end{equation}
\end{defn}
Observe that taking $q = 1$ recovers the original notion \eqref{eq-SFT} of an SFT-type almost complex structure (in this case $X_Q = R$.) Given $ \eta \in T_{(x,r)}( S \Sigma)$, write $  \eta = w + c Y$, where $ w \in T \Sigma$ and $c \in \mathbb{R}$. Setting 
\begin{equation}
\label{eq-components}
a = \frac{ \lambda_{(x,r)} ( \eta)}{Q(x,r)}, \qquad b = \frac{d Q(x,r) \eta}{Q(x,r)},
\end{equation}
and 
$$ v := w - a X_q(x),$$
we also have 
$$ \eta = a X_Q + b Y + \zeta_Q(v).$$
If $ \eta = a'X_Q + b'Y + \zeta_Q(v')$ is another tangent vector to $S \Sigma$ then 
\begin{align}
d \lambda( J_Q \eta , \eta ') & = (dr \wedge \alpha + r d \alpha)(J_Q(a X_Q + b Y + \zeta_Q(v)), a'X_Q + b' Y + \zeta_Q(v')) \nonumber \\
& = (dr \wedge \alpha + r d \alpha)(-b X_Q + a Y + \zeta_Q( jv)), a'X_Q + b' Y + \zeta_Q(v')) \nonumber \\ 
& = dr(-bX_Q + aY) \alpha(a'X_Q) - dr(a'X_Q + b'Y) \alpha( -bX_Q) + r d \alpha( jv, v')\nonumber \\
& = ( brdq(R) + ra) a'q - (-a' r dq(R) + rb') (-b q) + r d \alpha( jv,v') \nonumber \\
& = Q(aa' + bb') + r d \alpha ( jv,v'), \label{eq-norm}
\end{align}
which shows that $J_Q$ is compatible with $ d \lambda$. From now on we abbreviate 
$$ | \eta|_{J_Q}^2 := d \lambda(J_Q \eta, \eta) \ge 0.$$

Let us now fix a ``background'' Riemannian metric $g_{ \Sigma}$ on $ \Sigma$. We denote by $| \cdot|_{ \Sigma}$ the associated norm on $ T \Sigma$. Since $ \Sigma$ is compact, it follows from \eqref{eq-norm} that for any pair $(q,j)$ there exists a constant $ \varepsilon_{Q} >0$ such that for any $ \eta = w + cY \in T_{(x,r)} (S \Sigma)$, one has
\begin{equation}
\label{eq-norm-compare}
| \eta|^2_{J_Q}  \ge \varepsilon_{Q} r |w|^2_{ \Sigma}.
\end{equation}

\section{The extended maximum principle}
\label{sec-extended-max}
\begin{defn}
A connected non-compact $2n$-dimensional symplectic manifold $(M, \omega)$ is said to be \textbf{modelled on $ \Sigma$ at infinity} if the complement of a compact subset of $M$ is symplectomorphic to part of the symplectisation $ S \Sigma$ of $ \Sigma$. Explicitly, this means there exists a number $0 <  r_M <1$ and a compact domain $  M_* \subset M$ with compact closure such that $ M \setminus  
M_*$ is symplectomorphic to $ S\Sigma(r_M)$: 
$$
\iota : M \setminus  
M_*  \to S \Sigma(r_M), \qquad \iota^* (d\lambda) = \omega.
$$
 In particular, $ \partial M_* \cong \Sigma$, and outside of $ M_*$ the symplectic form $ \omega$ is exact (although this may not be the case on all of $M$.) More generally, a non-compact symplectic manifold is said to be \textbf{convex at infinity} if there exists a closed contact manifold $ \Sigma$ such that $M$ is modelled on $ \Sigma$ at infinity.\end{defn}
 
\begin{rem}
The assumption that $r_M < 1$ is of course, not important, and can always be achieved by rescaling. In particular, this notion is equivalent to the one discussed at the start of the Introduction.
\end{rem} 

Let us now fix a symplectic manifold $(M, \omega)$ which is modelled on $ \Sigma$ at infinity. In all of the following we will suppress the $ \iota$ from our notation. For $ r \ge r_M$ we denote by $ M_r := M \setminus S\Sigma(r)$, and think of $M$ as being the disjoint union $M_r \cup S\Sigma(r)$. Thus $M_* = M_{r_M}$.

\begin{defn}
\label{defn:ofH}
Let $\mathcal{H} \subset C^{\infty}(M \times S^1)$ denote the set of time-dependent Hamiltonians $H_t(z)$ on $M$ with the property that there exists $ r_M < r_H < 1$ such that the restriction of $H$ to $S \Sigma(r_H)$ is, up to a constant, 1-homogeneous in $r$:
\[
  H_t(x,r) = r h_t(x) + c_H, \qquad h_t :\Sigma \to \mathbb{R}, \qquad \forall \,(x,r) \in S \Sigma(r_H).
\]
\end{defn}
Just as with $X_Q$ above, the Hamiltonian vector field $X_{H_t}$ of $H_t$ satisfies
$$ X_{H_t}(x,r) = X_{h_t}(x) - dh_t(x)R(x), \qquad (x,r) \in S \Sigma(r_H).$$
where $X_h$ is the (now time-dependent) vector field on $ \Sigma$ defined by 
\begin{equation}
\label{eq-Xh}
\alpha (X_{h_t}) = h_t, \qquad d \alpha(X_{h_t}, \cdot) = dh_t(R) \alpha - d h_t.
\end{equation}
The vector field $ X_{h_t}$ generates a path $ \varphi_h^t : \Sigma \to \Sigma$ of contactomorphisms isotopic to the identity. We say that $h_t$ is the \textbf{contact Hamiltonian} associated to $ \varphi_h^t$. The \textbf{conformal factor} of $ \varphi_h^t$ is the function $ \kappa_h^t : \Sigma \to \mathbb{R}^+$ defined by 
$$ (\varphi_h^t)^* \alpha = \kappa_h^t  \alpha.$$
If $\kappa_h^t =1$ then we say $ \varphi_h^t$ is a \textbf{strict} contactomorphism. The flow $ \phi_H^t$ of the Hamiltonian vector field of $H_t$ is given by 
\begin{equation}
\label{eq-flow-XQ}
\phi_H^t(x,r) = \left( \varphi_h^t(x), \frac{r}{\kappa_h^t(x)} \right).
\end{equation}
For later use, let us note that if $Q = rq$ and $H = r h$ then the Poisson bracket is given by
\begin{align}
\{ Q, H\} & := dQ(X_H) \nonumber  \\
& = (qdr + r dq)(X_h - dh(R)Y) \nonumber \\
& =  - Q dh(R) + r dq(X_h). \label{eq-QH1}
\end{align}
From \eqref{eq-Xh} we have
\begin{align*}
  dq(X_h) &  = dq(R)h - d \alpha(X_q, X_h) \\
  & = dq(R)h + d \alpha(X_h, X_q) \\
  & = dq(R)h + dh(R)q - dh(X_q),
\end{align*}
and hence we also obtain
\begin{equation}
\label{eq-Poisson}
\{Q,H\} = H dq(R) - r dh(X_q).
\end{equation}
\begin{defn}
Let $ \mathcal{H}_{ \circ} $ denote the set of $H \in \mathcal{H}$ that in addition satisfy the following two conditions:
\begin{enumerate}
\item The contactomorphism $ \varphi_h^1$ has no fixed points.
\item All the fixed points of $ \phi_H^1$ are non-degenerate (i.e. for every fixed point $z \in M$ of $ \phi_H^1$, $1$ is not an eigenvalue of $D \phi_H^1(z):T_z M \to T_z M$).
\end{enumerate}
Note that by (1), any such fixed point is necessarily contained in the compact manifold $M_1$.
\end{defn}

\begin{defn}
\label{defn-JV}
Let $ \mathcal{J}$ denote the set of families $J = (J_t)$ of almost complex structures on $M$ which are compatible with $ \omega$ and which have the property that there exists $  r_M < r_J < 1$ and a smooth family $q_t$ of positive functions on $ \Sigma$ and an almost complex structure $j$ on $ \xi$ such that on $ S \Sigma (r_J)$, $J_t$ agrees with 
the corresponding family $J_{Q_t}$ of almost complex structures of twisted SFT-type.
\end{defn}

Let $\mathcal{L}(M) := C^{\infty}(S^1 ,M)$  denote the free loop space.  Given $J \in \mathcal{J}$ we define an $L^2$ scalar product $\left\llangle  \cdot , \cdot \right \rrangle_J$ on $\mathcal{L}(M)$ by 
\[
  \left\llangle \eta_1, \eta_2 \right\rrangle_J := \int_{S^1} \langle \eta_1(t), \eta_2(t) \rangle_{J_t} \,dt, \qquad v \in \mathcal{L}(M), \quad \eta_1, \eta_2 \in C^{\infty}(S^1, v^*TM),
\]
where $ \langle \cdot, \cdot \rangle_{J_t} := \omega(J_t \cdot, \cdot)$.
Given $H \in \mathcal{H}$  there is a well-defined one-form $\mathfrak{a}_{H}$ on the free loop space $\mathcal{L}(M)$, given for $v \in \mathcal{L}(M)$ by 
\[
  \mathfrak{a}_{H}(v) \cdot \eta = \int_{S^1} \omega_{v(t)}( \eta (t), \dot v (t) - X_{H_t}(v(t)))\,dt,  \qquad \eta \in C^{ \infty}( S^1, v^*TM).
\]
This one-form is typically not exact, but becomes exact when we lift it to an appropriate \textbf{Novikov cover} of $\mathcal{L}(M)$, c.f.  \cite{HoferSalamon1995}. We will not discuss this here, since it is not important as far as the maximum principle is concerned. The kernel of $ \mathfrak{a}_H$ is exactly  the one-periodic solutions of the Hamiltonian system determined by $H$:
\[
  \ker \mathfrak{a}_H = \left\{ v \in \mathcal{L}(M) \mid \dot v = X_H(v) \right\} \cong \mathrm{Fix}(\phi_H^1).
\]
We denote by $ \mathfrak{m}_{H,J}$ the vector field on $ \mathcal{L}(M)$ dual to $ \mathfrak{a}_H$: 
$$ \mathfrak{a}_{H}(v) \cdot \eta = \left\llangle \mathfrak{m}_{H,J}(v) , \eta \right\rrangle_J.$$
Explicitly, 
$$ \mathfrak{m}_{H,J}(v)(t) = J_t(v)( \dot v - X_{H_t}(v(t)).$$
The following lemma will be crucial for our maximum principle. 

\begin{lem}
\label{lem-missingbit}
Suppose $H \in \mathcal{H}_{ \circ}$ and $J \in \mathcal{J}$. There exists exists $ B,\varepsilon >0$ depending on $H$ and $J$ such that if a loop $v : S^1 \to M$ satisfies 
$$ \| \mathfrak{m}_{H,J}(v) \|_J < \varepsilon$$
then $v (S^1) \subset M_B$.
\end{lem}

\begin{proof}
By assumption on $S \Sigma(1)$ we can write $J_t = J_{Q_t}$ and $H_t = rh_t +c_H$. Set 
$$ \delta := \min_{(x,t) \in \Sigma \times S^1} q_t(x) , \qquad \Delta := \max_{(x,t) \in \Sigma \times S^1 } q_t(x),$$
so that $ 0 < \delta \le \Delta.$
Define
\begin{equation}
\label{eq-f}
m : [0,1] \times \Sigma \to \mathbb{R}, \qquad m(t,x) := \frac{ h_t(x)dq_t(x)R(x) - dh_t(x)X_{q_t}(x)}{q_t(x)}.
\end{equation}
By compactness, there exists $  c \le C$ such that 
$$ c \le m(t,x) \le C$$
for all $ (t,x)$.
Choose $ A,B > 1$ such that 
\begin{equation}
\label{eq-A-B-1}
B > \frac{ 1 + \Delta e^C  A}{\delta e^c};
\end{equation}
note in particular this implies $ B > \frac{\Delta}{\delta}A$.
Next, recall fixed background norm $ | \cdot|_{ \Sigma}$ defined just before \eqref{eq-norm-compare}. Since $ \varphi_h^1$ has no fixed points and $ \Sigma$ is compact, there exists $ \varepsilon_0$ such that 
$$ \int_{S^1} | \dot x - X_{h_t}(x)|^2_{ \Sigma}\,dt \ge \varepsilon_0$$
for any loop $x$ in $ \Sigma$. 
If $v(S^1)$ is entirely contained in $S \Sigma(A)$ then writing $v(t) = (x(t),r(t))$ and using \eqref{eq-norm-compare},
$$ \| \mathfrak{m}_{H,J}(v) \|_J^2 \ge \int_{S^1} \varepsilon_{Q_t} r(t) | \dot x(t) - X_{h_t}(x(t))|^2_{ \Sigma} \,dt \ge \varepsilon_A: = A \varepsilon_0 \max_{ t \in S^1} \varepsilon_{Q_t}.$$
If instead $v(S^1)$ is entirely contained in $M_B$ then since $M_B$ has compact closure and $H$ only has non-degenerate one-periodic orbits there exists $ \varepsilon_B$ such that if 
$$ \| \mathfrak{m}_{H,J}(v) \|_J \le \varepsilon_B$$
then $v(S^1)$ is contained in a small neighbourhood of the periodic orbits of $H$ (cf. \cite{Salamon1999}), and hence also in
$ M_1$. 
Set $$ \varepsilon_{A,B} := \min\{ \sqrt{ \varepsilon_A}, \varepsilon_B\}.$$ 
Then if $ \| \mathfrak{m}_{H,J}(v) \|_J < \varepsilon_{A,B}$ then there must exist an interval $[a,b] \subset S^1$ such that $ r(a) =A$ and $r(b) = B$, and such that $ r(t) \in (A,B)$ for all $ t \in (A,B)$. Let 
$$ \rho(t) := Q_t(v(t)) = q_t(x(t))r(t).$$
Then $ \rho(a) \le A \Delta$ and $ \rho(b) \ge \delta B$. Thus by \eqref{eq-A-B-1} there exists $[a',b'] \subset [a,b]$ such that 
$$ \rho(a') =  \Delta A, \qquad \rho(b') = \delta B.$$
Abbreviate $m(t) = m(t,x(t))$, where $m$ was defined in \eqref{eq-f}. We now look at the $Y$-component of $ \dot v - X_{H_t}(v)$. Using \eqref{eq-components} and \eqref{eq-Poisson} this is 
$$ \frac{dQ_t(v(t)) ( \dot v - X_{H_t}(v))}{Q_t(v(t)} = \frac{\dot \rho - \{Q, H\}}{\rho} = \frac{\dot \rho - m \rho}{\rho} .$$
Thus we can estimate 
$$ \| \mathfrak{m}_{H,J}(v)\|_J \ge \int_{a'}^{b'} \frac{1}{\sqrt{ \rho}} | \dot \rho +  m \rho| \,dt \ge \frac{1}{ \sqrt{\delta B}} \int_{a'}^{b'} | \dot  \rho + m \rho| \,dt$$
Let $$ n(t) := \exp \left( \int_0^t m(s)\,ds \right),$$
and
$$ s(t) := n(t)\rho(t),$$
so that $ \dot s = \dot n \rho + n \dot \rho  = n ( \dot \rho + m \rho)$.
Thus 
\begin{align*}
\| \mathfrak{m}_{H,J}(v)\|_J & \ge \frac{1}{e^C \sqrt{ \delta B}} \int_{a'}^{b'} | \dot s |\,dt \\ 
& =\frac{1}{e^C \sqrt{ \delta B}}(s(b') - s(a')) \\ 
& \ge \frac{1}{e^C \sqrt{ \delta B}}( e^c  \delta B - e^C \Delta A) \\
& \ge \frac{1}{e^C \sqrt{ \delta B}},
\end{align*}
where the last line used \eqref{eq-A-B-1}.
The lemma follows with $B$ as specified and
$$ \varepsilon := \min \left\{ \varepsilon_{A,B}, \frac{1}{e^C \sqrt{ \delta B}} \right\}.$$
This completes the proof.
\end{proof}

\begin{defn}
Let us say that a homotopy  $\mathbf{H} = (H^s) \subset  \mathcal{H} $, $s \in \mathbb{R}$ is a \textbf{continuation homotopy} of Hamiltonians if:
\begin{enumerate}
  \item $\mathbf{H} $ is asymptotically constant, i.e. 
  \[
  \begin{split}
    H^s = H^0,& \qquad \forall \, s \le 0\\
   H^s = H^1,& \qquad \forall \, s \ge 1.
  \end{split}
   \]
\item  The asymptotes $H^0,H^1$ both belong to $\mathcal{H}_{\circ}$.
  \item  $\mathbf{H} $ is monotonically increasing in $s$:
  \[
  \frac{\partial H^s}{\partial s} \ge 0.
  \]
\end{enumerate}
\end{defn}

\begin{defn}
Let us say that a homotopy  $\mathbf{J} = (J^s) \subset  \mathcal{J} $, $s \in \mathbb{R}$ is a \textbf{continuation homotopy} of almost complex structures if $J^s$ is asymptotically constant, i.e. 
\[
  \begin{split}
    J^s = J^0,& \qquad \forall \, s \le 0\\
   J^s = J^1,& \qquad \forall \, s \ge 1.
  \end{split}
   \]
\end{defn}

\begin{defn}
Given continuation homotopies $\mathbf{H}$ and $\mathbf{J}$ and a real number $E \ge 0$, we denote by
\[
  \mathcal{M}_{\mathbf{H}, \mathbf{J}}(E)
\]
the space of all smooth maps $u : \mathbb{R} \times S^1 \to M$ that are negative flow lines of the the $s$-dependent vector field $ \mathfrak{m}_{H^s,J^s}$:
\begin{equation}
\label{eq-neg-flow}
\partial_s u + \mathfrak{m}_{H^s,J^s}(u(s, \cdot)) = 0.
\end{equation}
and which have \textbf{energy}
\[
  \mathbb{E}_{\mathbf{J}}(u) := \int_{-\infty}^{\infty} \left\llangle \partial_s u, \partial_s u \right\rrangle_{J^s} \,ds \le E.
\]
\end{defn}
Explicitly, \eqref{eq-neg-flow} means that $u$ solves the partial differential equation:
\[
  \partial_s u + J_t(u) \partial_t u = J_t(u)X_{H_t^s}(u) .
\]

The main result of this section is the following \emph{a priori} $L^{\infty}$ bound on elements of $  \mathcal{M}_{\mathbf{H}, \mathbf{J}}(E)$. This result was stated in the Introduction as Theorem \ref{prop:emp}.
\begin{theorem}[The Extendend Maximum Principle]
\label{thm:maxprin}
Suppose we are given continuation homotopies $\mathbf{H}$ and $\mathbf{J}$ and $E \ge 0$. Then there exists a constant $r_0 =r_0(E, \mathbf{H}, \mathbf{J}) >0$ such that 
\[
 u(\mathbb{R} \times S^1) \cap S \Sigma(r_0)= \emptyset, \qquad \forall \,  u \in \mathcal{M}_{\mathbf{H}, \mathbf{J}}(E) .
\]
\end{theorem}

The proof will take some time, and we will proceed in several stages. Let us write $J_t = J_{Q^s_t}$ and $H_t = rh^t_s + c_{H^s}$ on $S \Sigma(1)$, where $h_t^s : \Sigma \to \mathbb{R}$ is a family of smooth functions such that $ \frac{ \partial h^t_s}{ \partial s} \ge 0$ and $q_t^s : \Sigma \to \mathbb{R}^+$ is a family of smooth positive functions. Both $h_t^s$ and $q_t^s$ only depend on $s$ in $[0,1]$.

First, we prove that a flow line can only spend a finite time outside of a compact set.

\begin{lem}
\label{lem-omega-bounded}
There exists $ T > 1$ and $L > 0$ such that if $ u \in \mathcal{M}_{ \mathbf{H}, \mathbf{J}}(E)$ and $ \Omega$ is a connected component of $ u^{-1}(S \Sigma(T))$ then $ \Omega$ is contained in $I \times S^1$ for $I \subset \mathbb{R}$ an interval of length at most $ L$.
\end{lem}

\begin{proof}
By Lemma \ref{lem-missingbit} there exists $ \varepsilon_i , B_i >0$ for $i = 0,1$ such that if a loop $v$ satisfies 
$$ \| \mathfrak{m}_{H^i, J^i}(v) \|_{J^i} < \varepsilon_i$$
for $i = 0,1$ then $v \subset M_{B_i}$. Let $ T := \max B_i$ and let $ \varepsilon = \min \varepsilon_i$.
Thus if $ u(s, \cdot) \subset S \Sigma(T)$ for all $s \in [a, b]$, where either $ [a,b] \subset ( - \infty, 0]$ of $[a,b] \subset [1, \infty)$, then we have 
$$ E = \int_{ \mathbb{R}} \| \partial_s u( s, \cdot) \|^2_{J^s} \,ds \ge |b-a| \varepsilon^2,$$
and hence $ |b-a| \le E \varepsilon^{-2}$.
Thus if $ \Omega $ is a connected component of $ u^{-1}(S \Sigma(T))$, decomposing $ \Omega$ as  
$$ \Omega = \Omega^- \cup \Omega_0 \cup \Omega^+,$$
where $ \Omega^- = \Omega \cap (- \infty ,0]$, $ \Omega_0 = \Omega \cup [0,1]$, and $ \Omega^+ = \Omega \cap [1, \infty)$ (some of which may be empty), we see that $ \Omega$ has length at most 
$$E \varepsilon^{-2} + 1 + E \varepsilon^{-2}.$$
Thus the lemma follows with $L = 2E \varepsilon^{-2}+1$.
\end{proof}
Let now $ \Omega$ denote a connected component of $ u^{-1}(S \Sigma(T))$. Write 
$$ u|_{ \Omega} = ( w , r),$$
so that $ w : \Omega \to \Sigma$ and $ r : \Omega \to (T, \infty)$. Our goal is to find a constant $ r_0 $ depending on $ \mathbf{H}, \mathbf{J}$ and $E$ (but not on $u$!) such that 
$$ \sup_{(s,t) \in \Omega}r(s,t) \le r_0.$$
Consider the function 
$$ \rho(s,t) :=Q_t^s(u(s,t)) =  q^s_t(w(s,t)) r(s,t).$$
It is also convenient to define the following two families of functions:

$$ f_t^s: \Sigma \to \mathbb{R}, \qquad g_t^s : \Sigma \to \mathbb{R},$$
where
$$f_t^s(x) :=\frac{\partial_s q_t^s (x)}{ q_t^s(x)},$$
and 
$$ g_t^s (x) := \frac{\partial_t q_t^s (x)}{ q_t^s(x)} + \left( \frac{dq_t^s(x)X_{h_t^s}(x))}{q_t^s(x)} - dh_t^s(x)(R(x)) \right).$$
Finally let $k_t^s$ denote the vector valued function
$$ k_t^s : \Sigma \to \mathbb{R}^2, \qquad k_t^s = (f^s_t, g^s_t).$$
In the following to keep the notation free from clutter we will drop the $s$ superscript and the $t$ subscript wherever possible. The key to our argument is the following elliptic differential inequality for $\log \rho$.
\begin{prop}
\label{prop:laplacian_of_mu}
Define
\[
  \mu : \Omega \to \mathbb{R}, \qquad \mu(s,t) := \log \rho(s,t). 
\]
Then $\mu$ satisfies the second order elliptic differential inequality
\[
  \Delta \mu  + k(w) \cdot \nabla \mu \ge \nabla k(w) \cdot w + k (w) \cdot \nabla w .
\]
\end{prop}
The proof of Proposition \ref{prop:laplacian_of_mu} is a variation on a  standard computation, see for instance \cite{Seidel2006} or \cite{AbbondandoloSchwarz2009}). We note however that Proposition \ref{prop:laplacian_of_mu} is not enough to apply the standard maximum principle, since there is no reason why the term on the right-hand side should be uniformly pointwise bounded.

\begin{proof}
On $ \Omega$ we have
\begin{align*}
\partial_s \rho & = \partial_s q (w) + dQ(u) \partial_s u \\
& = \partial_s q (w)r + dQ(u)( -J(u)( \partial_t u - X_H(u)) \\
& = \frac{\partial_s q (w)}{ q(w)} \rho - \lambda( \partial_t u) - \lambda(X_H(u)) \\
& = \frac{\partial_s q (w)}{ q(w)} \rho - \lambda( \partial_t u) + H(u).
\end{align*}
Similarly
\begin{align*}
\partial_t \rho & = \partial_t q (w) r + dQ(u) \partial_t u \\
& = \frac{\partial_t q (w)}{ q(w)} \rho + dQ(u) ( J(u) \partial_s u + X_H(u)) \\
& = \frac{\partial_t q (w)}{ q(w)} \rho + \lambda( \partial_s u) + \{ Q, H\}(u) \\
& = \frac{\partial_t q (w)}{ q(w)} \rho + \lambda( \partial_s u) + \left( \frac{dq(X_h(w))}{q(w)} - dh(R(w)) \right) \rho,
\end{align*}
where the last line used \eqref{eq-QH1}.Thus 
$$
dd^c \rho  = u^* \omega - \Big(  g(w) \partial_t \rho + f(w) \partial_s \rho + \big( \partial_t (g(w)) + \partial_s (f(w)) \big) \rho + \partial_s (H(u))\Big) ds \wedge dt.
$$
However
\[
\begin{split}
  u^*\omega & = \omega(\partial_s u, \partial_t u) ds \wedge dt \\
  &\omega\big( \partial_s u, J (u)\partial_s u  + X_H(u)\big) ds \wedge dt\\
& \Big(   - | \partial_s u|^2  + dH(u)(\partial_s u) \Big) ds \wedge dt
  \end{split}
\]
which means we can alternatively write
\[
   dd^c \rho =-  \Big( | \partial_s u|^2 +  g(w) \partial_t \rho + f(w) \partial_s \rho + \Big( \partial_t (g(w)) + \partial_s (f(w)) \Big) \rho + \partial_sH (u)\Big)ds \wedge dt.
\]
Since $d d^c \rho = - \Delta \rho ds \wedge dt$ we obtain
\begin{equation}
\label{eq:Laprho}
  \Delta \rho =    g(w) \partial_t \rho + f(w) \partial_s \rho + \Big( \partial_t (g(w)) + \partial_s (f(w)) \Big) \rho + \partial_sH (u) + | \partial_s u|^2 .
\end{equation}
Next, denoting by $ \langle  \cdot , \cdot \rangle$ the metric $ d \lambda(J_Q \cdot, \cdot)$, we have
\[
\begin{split}
  \left\langle \partial_s u , X_Q \right\rangle & = \left\langle -J(u)( \partial_t u - X_H(u) ), X_Q \right\rangle \\
  & = \left\langle  \partial_t u ,Y \right\rangle -  \{Q, H\}(u) \\
  & = \langle \frac{\partial_t r}{r} Y,Y \rangle - \{Q, H\}(u) \\
  & = \partial_t \rho  - g(w) \rho 
\end{split}
\]
where we used \eqref{eq-norm} and \eqref{eq-QH1}. Similarly 
\[
  \left\langle \partial_s u ,Y \right\rangle = \partial_s \rho  - f(w) \rho.
\]
The norm of $\partial_s u$ can be estimated from below by its projection onto the $(X_Q,Y)$-plane, which gives us
\begin{align*}
| \partial_s u |^2  & \ge \frac{1}{\rho}  \left\langle \partial_s u, X_Q \right\rangle^2 + \frac{1}{\rho} \left\langle  \partial_s u, Y \right\rangle^2 \\
& = \frac{1}{\rho} | \nabla \rho |^2 + \rho \Big( g(w)^2 + f(w)^2 \Big) - 2 \Big( \partial_t \rho \cdot g(w) +  \partial_s \rho \cdot f(w) \Big)  \\
& \ge \frac{1}{\rho} | \nabla \rho |^2 - 2 \Big( \partial_t \rho \cdot g(w) +  \partial_s \rho \cdot f(w) \Big).
\end{align*}
From the identity
\[
  \Delta \mu = \frac{1}{\rho} \Delta \rho - \frac{1}{\rho^2} |\nabla \rho|^2, 
\]
and using $ \partial_s H \ge 0$, we obtain from \eqref{eq:Laprho} that
\[
  \Delta \mu  + k(w) \cdot \nabla \mu  \ge \Big( \partial_s (f(w)) + \partial_t (g(w)) ) \Big),
\]
which is what we wanted to prove.
 \end{proof} 

 We can now complete the proof of Theorem \ref{thm:maxprin}.

\begin{proof}[Proof of Theorem \ref{thm:maxprin}]
 The Aleksandrov integral version of the weak maximum principle (\cite[Theorem 9.1]{GilbargTrudinger1983}, see also \cite[Appendix A]{AbbondandoloSchwarz2009}) tells us that
 \[
   \sup_{\Omega} \mu \le \sup_{ \partial \Omega} \mu   + C \| \nabla k(w) \cdot w + k(w) \cdot \nabla w \|_{L^2(\Omega)} = \log T + C \|\nabla k(w) \cdot w + k (w)\cdot \nabla w\|_{L^2(\Omega)} ,
 \]
 where $C$ depends on the $L^2$-norm of $k(w)$ and the diameter of $\Omega$. The diameter of $\Omega$ is bounded thanks to Lemma \ref{lem-omega-bounded}. Since $\Sigma$ is compact and $H^s$ and $J^s$ are asymptotically constant, it is clear $k(w)$ and $ \nabla k(w)$ are bounded in $L^2( \Omega)$. 
 
It remains to see that $ \nabla w$ is bounded in $L^2 ( \Omega)$. It suffices to show that $ \partial_s w$ and $ \partial_t w$ are bounded in $L^2$ using our fixed background norm $ | \cdot|_{ \Sigma}$, i.e. that
$$ \int_{ \Omega} | \partial_s w |_{ \Sigma}^2 \qquad \text{and} \qquad \int_{ \Omega} | \partial_t w|^2_{ \Sigma}$$
are bounded. That $ \partial_s w$ is bounded is easy, since from \eqref{eq-norm-compare} we have (temporarily writing all the sub/superscripts) 
$$ | \partial_s w(s,t)|^2_{ \Sigma} \le \frac{1}{\varepsilon_{Q^s_t} r(s,t)} | \partial_s u(s,t)|_{J_{Q^s_t}}^2.$$
Thus setting $ \varepsilon = \min_{s,t} \varepsilon_{Q^s_t}$, we have
$$ \int_{ \Omega} | \partial_s w|^2_{ \Sigma} \le \frac{1}{\varepsilon T } \int_{ \Omega} \| \partial_s u \|^2 \le \frac{E}{\varepsilon T}.$$
The $ \partial_t w$ term is a little trickier: for this we use the Floer equation to write $\partial_t u = J \partial_s u + X_{H}(u)$, which gives 
$$ \| \partial_t u \|^2 \le 2 \| \partial_s u \|^2 + 2 \| X_H (u) \|^2.$$
Then the same estimate gives 
\begin{align*}
| \partial_t w|^2_{ \Sigma} & \le \frac{1}{\varepsilon r} | \partial_t u|^2 \\
& = \frac{1}{\varepsilon r} | J \partial_s u + X_H(u)|^2 \\
& \le \frac{2}{\varepsilon r}\Big( | \partial_s u |^2 + | X_H(u)|^2 \Big) 
\end{align*}
From \eqref{eq-XbQ} and \eqref{eq-norm} we see that $ \frac{1}{r}|X_H(u)|^2$ is bounded, and this completes the proof. 
\end{proof} 

\section{Symplectic homology revisited}
\label{sec-proofs}
Let us recall the definition of the symplectic homology, $\mathrm{SH}_*(M)$ of an exact symplectic manifold $M$ modelled on a contact manifold $\Sigma$ at infinity. Our approach is slightly non-standard as we will use Hamiltonians which are of the form \eqref{eq:1homo_ham} at infinity (our main result, Theorem \ref{thm:maxprin} implies this is well-defined). For brevity we will work with Liouville manifolds only, rather than the technically more complicated setting of a  $\text{weak}^+$ monotone symplectic manifold as referred to in Theorem \ref{thm:sh_is_small}.

Let $h_t:\Sigma\to \mathbb{R}$ be a 1-periodic contact Hamiltonian such that the time-1 map $\varphi^1_h : \Sigma\to \Sigma$ of the corresponding contact isotopy has no fixed points. We call such contact Hamiltonians \textbf{admissible}.

Floer data for $h$ consists of a (time dependent) Hamiltonian $H\in\mathcal{H}_\circ$ and an almost complex structure $J\in\mathcal{J}$ such that (up to a constant) $H_t(x,r)=r h_t(x)$ at infinity, and such that the pair $(H,J)$ is \textbf{regular} in the sense that the linearisation of the Floer operator  
$$ u \mapsto \partial_s u + J_t(u) \partial_t u - J_t(u)X_{H_t}(u)$$ 
at any zero is surjective. Regularity is a generic condition, and it ensures that the moduli spaces of the solutions of the Floer equation are actually manifolds. In the situation above, we will say that $h$ is the \textbf{slope} of $H.$

Given Floer data $(H,J),$ one can construct a chain complex $\mathrm{CF}_*(H,J)$ which is generated (as a $\mathbb{Z}_2$ vector space) by 1-periodic Hamiltonian orbits of $H$ and graded by (negative) Conley-Zehnder index \cite[Section~2.4]{Salamon1999}. The homology of the chain complex $\mathrm{CF}_*(H,J)$ is denoted by $\mathrm{HF}_*(H,J),$ and called Floer homology. For more details on Floer homology, and proofs of the various asssertions made in these paragraphs, see for instance \cite{Salamon1999}.

The Floer homologies $\mathrm{HF}_*(H,J)$ and $\mathrm{HF}_*(H',J')$ associated to two differnet sets $(H,J)$ and $(H',J')$ for $h$ are canonically isomorphic. This leads to the Floer homology $\mathrm{HF}_*(h)$ which is associated to the admissible contact Hamiltonian $h.$ The dependence on the admissible contact Hamiltonian is essential, namely the Floer homologies corresponding to different contact Hamiltonians are in general not isomorphic. 

Given admissible contact Hamiltonians $h_t,h'_t:\Sigma\to \mathbb{R}$ such that $h_t \leq h'_t,$
using Theorem~\ref{thm:maxprin}, one can construct a \textbf{continuation map}
\[\mathrm{HF}_*(h)\to \mathrm{HF}_*(h').\]
The continuation maps turn the set $\{\mathrm{HF}_*(h) \mid h \text{ admissible}\}$ into a directed system of groups indexed by admissible contact Hamiltonians (with the standard order relation - pointwise $\leq$). The \textbf{symplectic homology} is the direct limit of this system
\[\mathrm{SH}_*(M):=\underset{h}{\lim_{\longrightarrow}}\: \mathrm{HF}_*(h).\]

Note that this definition coincides with the classical finite-slope definition of symplectic homology \cite{Viterbo1996}. This is due to constant admissible contact Hamiltonians form cofinal subset of the set of all admissible contact Hamiltonians.

Apart from continuation maps, there is yet another important class of morphisms between $\mathrm{HF}_*(h),$ the so called \textbf{naturality isomorphisms}. They are associated to the loops of Hamiltonian diffeomorphisms on $M.$ We first introduce some notation. Let $\phi_t:M\to M$ be an isotopy, let $H_t:M\to \mathbb{R}$ be a Hamiltonian, let $J_t$ be a family of almost complex structures on $M,$ and let $\gamma: S^1 \to M$ be a loop on M. Denote
\begin{align*}
&(\phi^\ast H)_t(x):= H_t(\phi_t(x)),\\
&(\phi^\ast J)_t:=\phi_t^\ast J_t,\\
&(\phi^\ast\gamma)(t):=\phi_t^{-1}(\gamma(t)).
\end{align*}

\begin{prop}
\label{prop:naturality}
Let $h_t:\Sigma\to \mathbb{R}$ be an admissible contact Hamiltonian, let $(H,J)$ be Floer data for $h,$ and let $G\in\mathcal{H}$ be a 1-periodic Hamiltonian that generates a loop of Hamiltonian diffeomorphisms. Denote by $g_t:\Sigma \to\mathbb{R}$ the slope of $G.$ Then, the contact Hamiltonian $ \tilde{h}_t:\Sigma\to \mathbb{R}$ defined by
\[ \tilde{h}_t(x):= \frac{(h_t-g_t)(\varphi^1_g(x))}{\kappa_g^t(x)}\]
is admissible, and $((\phi_K)^\ast H,(\phi_G)^\ast J)$ is Floer data for $\tilde{h}.$ The linear map
\[CF(H,J)\to CF((\phi_K)^\ast H,(\phi_G)^\ast J)\]
defined on generators by
\[\gamma \mapsto (\phi_G)^\ast\gamma\]
is an isomorphism. We denote the induced isomorphism on homology level by
\[\mathcal{N}(G)\::\: HF(h)\to HF(\tilde{h}),\]
and call it the naturality isomorphism.
\end{prop}
\begin{proof}
See the proof of \cite[Lemma~2.27]{Uljarevic2015}.
\end{proof}

Finally let us discuss the proofs of Theorem \ref{thm:infinite_order}. It is essentially identical to the argument from \cite[Theorem 4.13]{Uljarevic2015}, modulo use of the new maximum principle. Nevertheless for the reader's convenience we will provide a proof.

\begin{proof}[Proof of Theorem~\ref{thm:infinite_order}]
Let $h_t:\Sigma\to \mathbb{R}^+$, $t \in S^1$, be a positive, 1-periodic contact Hamiltonian that generates a loop $ \varphi_h^t$ of contactomorphisms. As in \eqref{eq:beta}, let $\beta:\mathbb{R}^+\to[0,1]$ be a smooth function 
$$
  \beta : \mathbb{R}^+ \to [0, 1], \qquad \beta(r) = \begin{cases}
  0, & r \in (0,1/4), \\
  1, & r \in (3/4, \infty),
  \end{cases}
   \qquad \beta'(r) \ge 0,
$$
Denote by $H_t:M\to \mathbb{R}$ the Hamiltonian that is equal to 0 on $M \setminus S \Sigma$ and to $(x,r)\mapsto\beta(r)rh_t(x)$ on $ S \Sigma$.
The time-1 map, $\phi_H^1$ represents the class $\Theta\left(\left[\varphi_h^t\right]\right)$ in the symplectic mapping class group. Assume, by contradiction, that the class $\left[\phi_H^1\right]$ is trivial. This means there exists a compactly supported Hamiltonian $L_t:M\to\mathbb{R}$ such that $\phi_H^1=\phi_L^1,$ and such that $L_{t+1}=L_t.$ Consider the Hamiltonian $G_t:M\to \mathbb{R}$ that generates the isotopy $\phi_H^t\circ\left(\phi_L^t\right)^{-1}.$ Note that $G$
generates a loop of Hamiltonian diffeomorphisms.
Denote by $G^m:M\to\mathbb{R}$ the Hamiltonian
\[G^m_t:=\left\{ \begin{matrix} m G_{mt}& t\in \left[0,\frac{1}{m}\right]\\ mG_{mt}& t\in\left[\frac{1}{m}, \frac{2}{m}\right]\\ \cdots& \\ m G_{mt}& t\in\left[ 1-\frac{1}{m}, 1\right],\end{matrix} \right.\]
and by $g^m$ the corresponding slopes. The isotopy of the Hamiltonian $G^m$ is equal to the concatenation of $m$ copies of $\{\phi_G^t\}.$
Let $\varepsilon>0$ be a small enough positive number. Denote by $f^m:\Sigma\to \mathbb{R}$ the admissible (cf. Proposition \ref{prop:naturality}) slope 
\[f^m_t(x):=\varepsilon\cdot \kappa^t_{g^m}\circ \left(\varphi^t_{g^m}\right)^{-1}+g^m_t.\]
The naturality with respect to $G^m$ provides the isomorphism
\begin{equation}\label{eq:natiso}
\mathrm{HF}(f^m)\to \mathrm{HF}(\varepsilon).
\end{equation}
\sloppy Hence the groups $\mathrm{HF}(f^m)$ are all isomorphic to the singular homology $\mathrm{H}(M_1,\Sigma;\mathbb{Z}_2).$
Since $\min f^m \ge m\cdot\min h, $ and since $\min h>0, $ for each admissible $a\in\mathbb{R}$ the map
\[\mathrm{HF}_*(a)\to \mathrm{SH}_*(M;\mathbb{Z}_2)\]
factors through $\mathrm{HF}_*(f^m)$ for large enough $m.$ Hence 
\[\dim \mathrm{SH}(M;\mathbb{Z}_2)\leq \dim \mathrm{H}(M_1,\Sigma;\mathbb{Z}_2)=\dim \mathrm{H}(M;\mathbb{Z}_2).\]
Contradiction! Therefore, the class $\Theta\left(\left[\varphi_h^t\right]\right)$ is not trivial. Since the iterates of the loop $\{\varphi^t_h\}$ are generated by positive contact Hamiltonians as well, the same argument shows that the class $\Theta\left(\left[\varphi_h^t\right]\right)$ is in fact of infinite order in the symplectic mapping class group.
\end{proof}

\section{Translated points and the positive symplectic homology of a contactomorphism}
\label{sec:tp}
\begin{defn}
Let $ ( \Sigma , \alpha)$ be a co-oriented contact manifold with a choice of contact form. Let $ \theta^t : \Sigma \to \Sigma$ denote the Reeb flow of $ \alpha$. Let $ \varphi : \Sigma \to \Sigma$ denote a contactomorphism. Let $ \kappa : \Sigma \to \mathbb{R}^+$ denote the conformal factor of $ \varphi$ and $ \alpha$, so that $ \varphi^* \alpha = \kappa \alpha$. A \textbf{translated point} of $ \varphi$ and $ \alpha$ is a point $ x \in \Sigma$ such that there exists $ T \in \mathbb{R}$ such that
$$ \varphi(x) = \theta^{-T}(x), \qquad \kappa(x) =1.$$
We call the minimal (in absolute value) such $T$ the \textbf{period} of $ x$ (the word ``minimal'' is included to deal with the case where the translated point lies on a closed Reeb orbit of $ \alpha$.) This notion was introduced by Sandon in \cite{Sandon2012}, although it can be seen as a special case of a leaf-wise intersection point as introduced by Moser in \cite{Moser1978}.
\end{defn}

In this section we show how the maximum principle proved here can be used to prove that if a Liovuille domain has infinite-dimensional symplectic homology then any contactomorphism of the boundary  isotopic to the identity through contactomorphisms has infinitely many translated points. This result is not new; in \cite{AlbersMerry2013a} the same result was obtained using Rabinowitz Floer homology (and symplectic homology is infinite dimensional if and only if Rabinowitz Floer homology is, c.f. \cite{CieliebakFrauenfelderOancea2010}). We then go on to explain how by working with the analogue of \textbf{positive} symplectic homology, we can define a more refined invariant which we call the positive symplectic homology of a contactomorphism. Its properties will be elucidated in a sequel to the present paper.
\newline

Let $(M, d \lambda)$ denote an exact symplectic manifold modelled on $ \Sigma$ at infinity as before. Let $ \varphi^t : \Sigma \to \Sigma$ be a path of contactomorphisms with $ \varphi_0 = \mathrm{Id}$, and let $ h_t : \Sigma \to \mathbb{R}$ denote the contact Hamiltonian associated to $ \varphi^t$. 

After reparametrising $ \varphi^t$ if necessary, we may assume $h_t$ is 1-periodic in $t$. Let $ \kappa^t : \Sigma \to \mathbb{R}^+$ denote the conformal factor of $ \varphi^t$. Denote by $ \theta^t : \Sigma  \to \Sigma$ the Reeb flow of $ \alpha := \lambda|_{ \Sigma}$. 

Fix a small $ \varepsilon > 0$. Let $f^c$ denote the function on $ \Sigma$ given by
$$ f^c_t(x) = h_t(x) + c \kappa_t(x),$$
and let $F^c$ denote a Hamiltonian on $M$ such that 
\begin{equation}
\label{Fc}
F^c_t (x) =  rf^c_t(x), \qquad (x,r) \in \Sigma \times ( \varepsilon , \infty).
\end{equation}
The 1-periodic orbits of this Hamiltonian contained in $ \Sigma \times ( \varepsilon , \infty)$ correspond to translated points of $ \varphi^1$ with period $c$. Indeed, the Hamiltonian flow of $F^c_t$ is given by 
$$ \phi_{F^c}^t( x, r) = \left( \varphi^t( \theta^{ct}(x)), \frac{r}{\kappa^t( \theta^{ct}(x))} \right),$$
and hence if $ \phi_{F^c}^1(x,r) = (x,r)$ then $ \theta^c(x)$ is a translated point of $ \varphi^1$ with period $c$.

In particular, if $c$ is not the period of a translated point of $ \varphi^1$ then $f^c$ is admissible, and hence $ \mathrm{HF}(f^c)$ is well-defined. Since $ \kappa^t$ is positive, if $ c \le c'$ then $f^c \le f^{c'}$, and hence there is a well-defined continuation map 

$$ \mathrm{HF}(f^c) \to \mathrm{HF}(f^{c'}).$$

This gives almost immediately the following result.

\begin{prop}
\label{proptp}
Suppose $ \mathrm{SH}_*(M)$ is infinite dimensional. Then any contactomorphism of the boundary isotopic to the identity through contactomorphisms either has infinitely many translated points or a translated point on a closed Reeb orbit.
\end{prop}

\begin{proof}
If not, then there exists a maximum $C $ such that $ \varphi^1$ has no translated points with period greater than $c$. The usual continuation arguments then show that 
$$ \mathrm{HF}(f^c) \cong \mathrm{HF}(f^{c'}), \qquad \forall\, C < c < c'.$$
Since $ \mathrm{HF}(f^c)$ is finitely generated for any finite $c$ (by assumption), it follows that 
$$\underset{c}{\lim_{\longrightarrow}}\: \mathrm{HF}(f^c)$$
is finite-dimensional. But since $f^c$ forms a cofinal family, this direct limit agrees with $ \mathrm{SH}(M)$; a contradiction.
\end{proof}

The group $$\underset{c}{\lim_{\longrightarrow}}\: \mathrm{HF}(f^c)$$ is not so interesting (as an invariant of $ \varphi^t$) since it just agrees with the symplectic homology of $M$. But we can rectify this as follows. Suppose $ \varphi^1$ has no fixed points, so that $ \mathrm{HF}(h)$ is well-defined. Let $ p >0$ be such that any translated point $x$ of $ \varphi^1$ with period $T>0$ has $ T > p$ (such $p$ exists as $ \varphi^1$ has no fixed points.)

Let $H_t$ be a Hamiltonian on $M$ which agrees with $rh_t(x)$ on $ \Sigma \times ( \varepsilon , \infty)$ and is $C^2$ small and Morse on the rest of $M$. Let $ A > 0$ be such that any critical point $z(t)$ of the action functional  
$$ \mathcal{A}_{H} (z) := \int_{S^1}z^* \lambda - \int_{S^1}H(z) \,dt.$$
has action $ \mathcal{A}_{H}(z) \le A$ (this is well defined since $ \varphi^1$ has no fixed points .)

Fix now a number $ c  > p$ which is not the period of a translated point of $ \varphi^1$ and choose $\varepsilon < r_0 < r_1$. Consider a function $ g : \mathbb{R}^+ \to \mathbb{R}^+$ which satisfies
$$ g(r) = 0, \qquad \forall\, 0 \le r \le r_0, \qquad \text{and} \qquad  g'(r) = c, \qquad \forall \, r \ge r_1,$$
and
$$ rg'(r) - g(r) > 0, \qquad \forall \, r > r_0.$$
We now modify the function $ F^c$ from \eqref{Fc} to a new function $\tilde{F}^c : M \to \mathbb{R}$ by requiring that
$$ \tilde{F}^c_t(x,r) = rh_t(x) + g( \kappa_t(x)) \qquad \forall \, (x,r) \in \Sigma \times ( \varepsilon , \infty).$$
For an appropriate choice of $r_0 < r_1$ (depending on $p$ and $A$), the critical points of $ \mathcal{A}_{\tilde{F}^c}$ come in two forms:
\begin{itemize}
  \item Critical points $z(t)$ corresponding to translated points of $ \varphi^1$ with period less than $c$ and action $ \mathcal{A}_{\tilde{F}^c}(z) > A$.
  \item Critical points $z(t)$ of $ \mathcal{A}_{H}$ with action $\mathcal{A}_{\tilde{F}^c}(z) \le A$.
\end{itemize}
Let $ \mathrm{CF}^+( f^c)$ denote the subcomplex generated by those quotienting out those critical points with action less or equal to $A$. Let $ \mathrm{HF}^+(f^c)$ denote the associated homology. Thus by construction this homology is generated by the translated points of $ \varphi^1$ with period less than $c$. Now set 
$$ \mathrm{SH}^+( \varphi^t, M) := \underset{c}{\lim_{\longrightarrow}}\: \mathrm{HF}^+(f^c).$$
We call $ \mathrm{SH}^+( \varphi^t, M)$ the \textbf{positive symplectic homology for the contact isotopy $ \varphi^t$}. It is defined for any isotopy with $ \varphi^1 \ne \mathrm{Id}$. By construction we obtain a long exact sequence
$$ \dots \to \mathrm{HF}_*(h) \to \mathrm{SH}^+_*( \varphi^t, M) \to \mathrm{SH}_*(M) \to \mathrm{HF}_{*-1}(h) \to \dots $$
If $ \varphi^t = \theta^{ a t}$ is a piece of Reeb flow, where $ a>0$ is small, then $$ \mathrm{HF}_*(h) = \mathrm{H}_{*+n}(M, \Sigma)$$ and thus $ \mathrm{SH}^+( \varphi^t, M)$ agrees with the \textbf{positive symplectic homology} $ \mathrm{SH}^+(M)$ of $M$. But in general this is not true. Up to a grading shift, these groups only depend on the time-1 map $ \varphi^1$. The groups $ \mathrm{SH}^+( \varphi^t, M)$ can be used to give more interesting results on the existence of translated points than Proposition \ref{proptp}, and we will continue its study in this paper's sequel.
\bibliographystyle{plain}
\bibliography{bib}
\end{document}